\documentclass[10pt]{amsart}
\usepackage{amsmath,amssymb,latexsym,soul,cite,amsthm,color,enumitem,graphicx,tikz, mathtools,microtype}
\usepackage[colorlinks=true,urlcolor=cobalt,citecolor=cobalt,linkcolor=cobalt,linktocpage,pdfpagelabels,bookmarksnumbered,bookmarksopen]{hyperref}
\definecolor{cobalt}{rgb}{0.0, 0.28, 0.67}
\usepackage[english]{babel}
\usepackage[left=2.5cm,right=2.5cm,top=2.5cm,bottom=2.5cm]{geometry}

\numberwithin{equation}{section}

\newtheorem{Teo}{Theorem}[section]
\theoremstyle{plain}

\theoremstyle{plain}
\newtheorem{Pro}[Teo]{Proposition}
\theoremstyle{plain}
\newtheorem{Cor}[Teo]{Corollary}
\newtheorem{Def}[Teo]{Definition}
\theoremstyle{definition}
\newtheorem{Oss}[Teo]{Remark}

\newcommand{\N}{{\mathbb N}}
\newcommand{\Z}{{\mathbb Z}}
\newcommand{\R}{{\mathbb R}}

\newcommand{\eps}{\varepsilon}

\newenvironment{enumroman}{\begin{enumerate}

}{\end{enumerate}}

\title[Non-local monotonicity and u.c.p.]{Strict monotonicity and unique continuation for general non-local eigenvalue problems}

\author[S.\ Frassu, A.\ Iannizzotto]{Silvia Frassu, Antonio Iannizzotto}

\address[S.\ Frassu, A.\ Iannizzotto]{Department of Mathematics and Computer Science
\newline\indent
University of Cagliari
\newline\indent
Viale L.\ Merello 92, 09123 Cagliari, Italy}
\email{silvia.frassu@unica.it, antonio.iannizzotto@unica.it}


\begin{document}

\begin{abstract}
We consider the weighted eigenvalue problem for a general non-local pseudo-differential operator, depending on a bounded weight function. For such problem, we prove that strict (decreasing) monotonicity of the eigenvalues with respect to the weight function is equivalent to the unique continuation property of eigenfunctions. In addition, we discuss some unique continuation results for the special case of the fractional Laplacian.
\end{abstract}

\maketitle

\begin{center}
Version of \today\
\end{center}

\section{Introduction}\label{sec1}

\noindent
Weighted eigenvalue problems can be studied for any type of linear elliptic (ordinary or partial) differential operator or even for integro-differential operator, exhibiting some kind of uniform ellipticity, and under various boundary conditions. In most cases, the resulting problem can be written as
\[\begin{cases}
Lu=\lambda\rho(x)u & \text{in $\Omega$} \\
u\in X(\Omega),
\end{cases}\]
where $L$ is the chosen operator, $\Omega$ is a bounded domain, $\rho\in L^\infty(\Omega)$ is the weight function, and $X(\Omega)$ is some function space defined on $\Omega$ (which includes the boundary conditions). The problem above admits a sequence of variational eigenvalues, generally unbounded both from above and below (the sequence is bounded from below if $\rho$ is non-negative, and from above if $\rho$ is non-positive), denoted by
\[\ldots \le\lambda_{-k}(\rho)\le\ldots\le\lambda_{-1}(\rho)<0<\lambda_1(\rho)\le\ldots\le\lambda_k(\rho)\le\ldots\]
(we refer to \cite{F}). Even non-linear operators, under some homogeneity and monotonicity properties, exhibit an analogous sequence of variational eigenvalues, though it is not known whether they cover the whole spectrum (see \cite{PAO}).
\vskip2pt
\noindent
Clearly, every eigenvalue depends on the weight function, and it is an easy consequence of the variational characterization of eigenvalues that the mapping $\rho\mapsto\lambda_k(\rho)$ is monotone non-increasing for all integer $k\neq 0$, with respect to the pointwise order in $L^\infty(\Omega)$. A more delicate question is whether such dependence is {\em strictly decreasing}. De Figueiredo and Gossez \cite{FG} have proved that, if $L$ is a second order elliptic operator with bounded coefficients and Dirichlet boundary conditions, strict monotonicity of the eigenvalues with respect to the weight is equivalent to the {\em unique continuation property} (for short, u.c.p.) of eigenfunctions, i.e., to the fact that eigenfunctions vanish at most in a negligible set. The result strongly relies on min-max characterizations of the eigenvalues of both signs.
\vskip2pt
\noindent
Such equivalence is extremely important in the study of {\em non-linear} boundary value problems of the type
\[\begin{cases}
Lu=f(x,u) & \text{in $\Omega$} \\
u\in X(\Omega),
\end{cases}\]
where $f:\Omega\times\R\to\R$ is a Carath\'eodory mapping, asymptotically linear in the second variable either at zero or at infinity. Many existence/multiplicity results for non-linear boundary value problems are obtained by locating the limits
\[\lim_{t\to 0,\infty}\frac{f(x,t)}{t}\]
in known spectral intervals of the type $[\lambda_k(\rho),\lambda_{k+1}(\rho)]$, possibly involving several weight functions, and then by using strict monotonicity to avoid resonance phenomena. Thus, it is possible to compute the critical groups of the corresponding energy functional at zero and at infinity, and so deduce the existence of non-trivial solutions (one typical application of this approach for the fractional Laplacian can be found in \cite{IP}).
\vskip2pt
\noindent
Motivated by the considerations above, we devote this note to proving an analog of the results of \cite{FG} for a very general family of linear non-local operators, introduced by Servadei and Valdinoci in \cite{SV}, which includes as a special case the fractional Laplacian (for a general discussion on fractional boundary value problems, we refer to \cite{MBRS}). We study the following eigenvalue problem:
\begin{equation}
\begin{cases}
L_K u = \lambda \rho(x) u  & \text{in $\Omega$ } \\
u = 0 & \text{in $\R^N \setminus \Omega$.}
\end{cases}
\label{P}
\end{equation}
Here $\Omega\subset \R^N$ is a bounded domain with a Lipschitz continuous boundary, the leading operator is defined by
$$L_K u(x)= P.V. \int_{\R^N} (u(x)-u(y))K(x-y)\,dy,$$  
namely a general non-local operator, whose kernel $K$ satisfies the following hypotheses:
\begin{itemize}[leftmargin=0.75cm]
	\item[${\bf H}_K$] $K:\R^N \setminus \{0\} \to (0, +\infty)$ s.t.\
	\vspace{0.1cm}
	\begin{enumroman}
		\item\label{h1} $m K \in L^1(\R^N)$, where $m(x)=\min\{|x|^2,1\}$;
		\vspace{0.1cm}
		\item\label{h2} $K(x)\geq \alpha |x|^{-(N+2s)}$ in $\R^N\setminus \{0\} \; (\alpha>0, s \in (0,1) \; \text{s.t.} \; N>2s)$;
		\vspace{0.1cm}
		\item\label{h3} $K(-x)=K(x)$ in $\R^N \setminus \{0\}$.
	\end{enumroman}
\end{itemize}
For $K(x)= |x|^{-N-2s}$ we have $L_K=(-\Delta)^s$ (the Dirichlet fractional Laplacian). The operator $L_K$ is the infinitesimal generator of a (possibly anisotropic) L\'evy process, and thus it arises often in modeling phenomena of anomalous diffusion with long distance interactions (see \cite{F1,RO} and the references therein). Problem \eqref{P} depends on a weight function $\rho \in L^{\infty}(\Omega)$, and it admits a sequence of eigenvalues $(\lambda_k(\rho))_{k \in \Z_0}$ ($k \in \pm \N_0$ if $\rho$ has constant sign). Here we prove equivalence between the strict monotonicity of the mapping $\rho\mapsto\lambda_k(\rho)$ ($k\in\Z_0$), and u.c.p.\ of eigenfunctions. We note that, in general, u.c.p.\ for solutions of non-local problems is a challenging open problem, though some partial results have been established, mostly regarding the case of the fractional Laplacian.
\vskip2pt
\noindent
In Section \ref{sec2} we give problem \eqref{P} an appropriate functional analytic setting and recall the general structure of the spectrum; in Section \ref{sec3} we prove our equivalence result; and in Section \ref{sec4} we survey some known results about u.c.p.\ for non-local operators.
\vskip4pt
\noindent
{\bf Notation.} For all $U\subset\R^N$ we denote by $|U|$ its Lebesgue measure. For any two measurable functions $f$, $g$ defined in $U$, we write $f\le g$ for '$f(x)\le g(x)$ for a.e.\ $x\in U$', and similarly $f\ge g$, $f<g$, $f>g$, and $f\equiv g$. We denote by $f^+$, $f^-$ the positive and negative parts of $f$, respectively. For all $q\in[1,\infty]$ we denote by $\|\cdot\|_q$ the norm of $L^q(\Omega)$.

\section{Functional analytic setting and general properties of the eigenvalues}\label{sec2}

\noindent
We introduce a functional analytic setting for problem \eqref{P}, following \cite{SV} (see also \cite{MBRS}). For all measurable $u:\R^N \rightarrow \mathbb{R}$ set 
$$[u]_{K}^2 := \int_{\R^{2N}} (u(x)-u(y))^2 K(x-y)\,dxdy,$$
then define
$$X_K(\Omega)=\big\{u \in L^2(\R^N): [u]_{K} < \infty,\,u=0 \text{ a.e.\ in } \R^N \setminus \Omega \big\},$$
endowed with the scalar product
$$\left\langle u,v \right\rangle = \int_{\R^{2N}} (u(x)-u(y))(v(x)-v(y))K(x-y)\,dxdy$$
and the corresponding norm $\|u\|=[u]_{K}$. Then, $(X_K(\Omega),\|\cdot\|)$ is a  Hilbert space, continuously embedded into the fractional Sobolev space $H^s(\Omega)$ and hence into $L^q(\Omega)$ for all $q\in [1,2_s^*]$ ($2_{s}^{*}:= 2N/(N-2s)$), with compact embedding iff $q<2_s^*$ \cite[Lemmas 5-8]{SV}.
\vskip2pt
\noindent
We say that $u \in X_K(\Omega)$ is a (weak) solution of \eqref{P}, if for all $v \in X_K(\Omega)$
$$\left\langle u,v \right\rangle = \lambda \int_{\Omega} \rho(x) uv\,dx.$$
If, for a given $\lambda\in\R$, problem \eqref{P} has a non-trivial solution $u\in X_K(\Omega)\setminus\{0\}$, then $\lambda$ is an {\em eigenvalue} with associated {\em eigenfunction} $u$. The {\em spectrum} of \eqref{P} is the set of all eigenvalues, denoted $\sigma(\rho)$.
\vskip2pt
\noindent
Following the general scheme of \cite{F}, we provide a characterization of $\sigma(\rho)$. In particular, we provide four min-max formulas for eigenvalues of both signs, that will be a precious tool in the proof of our main results:

\begin{Pro}\label{EV}
Let $\rho \in L^{\infty}(\Omega)$, $\rho \not\equiv 0$. Set for all integer $k>0$
$$\mathcal{F}_k=\big\{F \subset X_K(\Omega): F \text{ linear subspace,}\,\mathrm{dim}(F)=k\big\},$$
and
\begin{equation}
\lambda_k^{-1}(\rho)= \sup_{F \in \mathcal{F}_k} \inf_{u \in F,\,||u||=1}  \int_{\Omega} \rho(x) u^2\,dx
= \inf_{F \in \mathcal{F}_{k-1}} \sup_{u \in F^{\perp},\,||u||=1}  \int_{\Omega} \rho(x) u^2\,dx, 
\label{E+}
\end{equation}  
\begin{equation}
 \lambda_{-k}^{-1}(\rho)= \inf_{F \in \mathcal{F}_k} \sup_{u \in F,\,||u||=1}  \int_{\Omega} \rho(x) u^2\,dx
 = \sup_{F \in \mathcal{F}_{k-1}} \inf_{u \in F^{\perp},\,||u||=1}  \int_{\Omega} \rho(x) u^2\,dx, 
\label{E-}
\end{equation}  
Then,
\begin{enumroman}
\item if $\rho^+\not\equiv 0$, then
\[0<\lambda_1(\rho)<\lambda_2(\rho)\le\ldots \lambda_k(\rho)\le\ldots\to \infty,\]
and for all $k\in\N_0$ $\lambda_k(\rho)$ is an eigenvalue of \eqref{P} with associated eigenfunction $e_{k,\rho}\in X_K(\Omega)$;
\item if $\rho^-\not\equiv 0$, then
\[0>\lambda_{-1}(\rho)>\lambda_{-2}(\rho)\ge\ldots\ge\lambda_{-k}(\rho)\ge\ldots\to -\infty,\]
and for all $k\in\N_0$ $\lambda_{-k}(\rho)$ is an eigenvalue of \eqref{P} with associated eigenfunction $e_{-k,\rho}\in X_K(\Omega)$.
\end{enumroman}
Moreover, all sup's and inf's in \eqref{E+}, \eqref{E-} are attained (at $\lambda_{\pm k}(\rho)$-eigenfunctions). If $\rho\ge 0$ (resp.\ $\rho\le 0$), then \eqref{P} admits only positive (resp.\ negative) eigenvalues. Finally, for all $k,h\in\Z_0$
$$\left\langle e_{k,\rho},e_{h,\rho} \right\rangle=\delta_{kh}.$$
\end{Pro}
\begin{proof}
Given $u \in X_K(\Omega)$, the linear functional 
$$v \mapsto  \int_{\Omega} \rho(x) uv\,dx$$
is bounded in $X_K(\Omega)$. By Riesz' representation theorem there exists a unique $T(u) \in X_K(\Omega)$ s.t.\
$$\left\langle T(u),v \right\rangle= \int_{\Omega} \rho(x) uv\,dx.$$
So we define a bounded linear operator $T\in\mathcal{L}(X_K(\Omega))$, indeed for all $u \in X_K(\Omega)$
$$||T(u)||=\sup_{v \in X_K(\Omega),\,||v||=1} \Big| \int_{\Omega} \rho(x) uv\,dx\Big| \leq ||\rho||_{\infty} ||u||_2 \sup_{v \in X_K(\Omega),\,||v||=1}	||v||_2 \leq C ||u||.$$
Clearly $T$ is symmetric. Moreover, $T$ is compact. Indeed, let $(u_n)$ be a bounded sequence in $X_K(\Omega)$, then (passing to a subsequence) $u_n \rightharpoonup u$ in $X_K(\Omega)$, 
$u_n \rightarrow u$ in $L^2(\Omega)$. So  we have for all $v \in X_K(\Omega)$, $||v||\leq 1$	
$$\big|\left\langle T(u_n)-T(u),v\right\rangle\big|\leq  \int_{\Omega} |\rho(x) (u_n-u)v|\,dx\leq ||\rho||_{\infty} ||u_n-u||_2 ||v||_2 \leq C ||u_n-u||_2,$$
and the latter tends to $0$ as $n\to\infty$. So $T(u_n)\rightarrow T(u)$ in $X_K(\Omega)$. First assume $\rho^+\not\equiv 0$, then
$$\mu_1=\sup_{u \in X_K(\Omega), ||u||=1} \left\langle T(u),u\right\rangle >0.$$
By \cite[Lemma 1.1]{F}, there exists $e_{1,\rho} \in X_K(\Omega)$ s.t.\ $T(e_{1,\rho})=\mu_1 e_{1,\rho}$, $||e_{1,\rho}||=1$. Further, set for all $k>0$
$$\mu_k=\sup_{F \in \mathcal{F}_k} \inf_{u \in F, ||u||=1}  \int_{\Omega} \rho(x) u^2\,dx>0.$$
Then, by \cite[Propositions 1.3, 1.8]{F}, there exists $e_{k,\rho} \in X_K(\Omega)$ s.t.\ $T(e_{k,\rho})=\mu_k e_{k,\rho}$. Applying \cite[Lemma 1.4]{F}, we see that $(\mu_k)$ is a sequence of eigenvalues of $T$, s.t.\ $\mu_k \ge \mu_{k+1}$ and $\mu_k \rightarrow 0^+$. Besides, for all $k>0$, the eigenspace associated to $\mu_k$ has finite dimension (hence it admits an orthonormal basis). So, by relabeling $(e_{k,\rho})$ if necessary, we have for all $k,h\in\Z_0$
$$\left\langle e_{k,\rho},e_{h,\rho} \right\rangle=\delta_{kh},$$
which in turn implies for all $k\neq h$
$$ \int_{\Omega} \rho(x) e_{k,\rho},e_{h,\rho}\,dx = 0.$$
Now set $\lambda_k(\rho)=\mu_k^{-1}$. Then, \eqref{E+} follows from the definition of $\mu_k$ and \cite[Proposition 1.7]{F}. Besides, we have for all $v\in X_K(\Omega)$
$$\left\langle e_{k,\rho},v\right\rangle  = \lambda_k(\rho) \int_{\Omega} \rho(x) e_{k,\rho} v\,dx,$$
so $\lambda_k(\rho) \in \sigma(\rho)$ with associated eigenfunction $e_{k,\rho}$. Moreover, $\lambda_k(\rho)\to \infty$ as $k\to\infty$, all eigenspaces are finite-dimensional, and eigenfunctions associated to different eigenvalues are orthogonal. Also, all sup's and inf's in \eqref{E+} are attained at (subspaces generated by) eigenfunctions. Finally, reasoning as in \cite[Proposition 2.8]{IP} it is easily seen that $\lambda_1(\rho)<\lambda_2(\rho)$ and that there are no positive eigenvalues other than $\lambda_k(\rho)$, $k>0$.
\vskip2pt
\noindent
Similarly, if $\rho^- \not\equiv 0$, then \eqref{E-} defines a sequence $(\lambda_{-k}(\rho))$ of negative eigenvalues of \eqref{P} s.t.\ $\lambda_{-k} (\rho) \rightarrow -\infty$, with an orthonormal sequence $(e_{-k,\rho})$ of associated eigenfunctions.
\vskip2pt
\noindent
By \cite[Proposition 1.11]{F}, if $\rho \geq 0$ there are no negative eigenvalues, similarly if $\rho \leq 0$ there are no positive eigenvalues.
\end{proof}

\noindent Now we prove continuous dependence of the eigenvalues on $\rho$, with respect to the norm topology of $L^{\infty}(\Omega)$ (in the forthcoming results, we say that $k\in\Z_0$ is {\em admissible} if the corresponding eigenvalue does exist):

\begin{Pro} \label{CD}
Let $(\rho_n)$ be a sequence in $L^{\infty}(\Omega)$ s.t.\ $\rho_n \rightarrow \rho$ in $L^{\infty}(\Omega)$. Then, for all admissible $k\in \Z_0$ we have $\lambda_k(\rho_n) \rightarrow \lambda_k(\rho)$.	
\end{Pro}

\begin{proof}
For simplicity, assume $\rho_n^+ \not\equiv 0$ for all $n \in \N$, $\rho^+ \not\equiv 0$, and $k>0$ (other cases are studied similarly). Set for all $u,v \in X_K(\Omega)$
$$\left\langle T_n(u),v\right\rangle  = \int_{\Omega}  \rho_n(x) uv\,dx,$$
then $T_n \in \mathcal{L}(X_K(\Omega))$ is a bounded, symmetric, compact operator. Similarly we define $T \in \mathcal{L}(X_K(\Omega))$ using $\rho$. We claim that
\begin{equation}
T_n \rightarrow T \ \text{ in } \mathcal{L}(X_K(\Omega)).
\label{T}
\end{equation}
Indeed, for any $n \in \N$ and $u \in X_K(\Omega)$, $||u||=1$, we have by the Cauchy-Schwarz inequality
\begin{align*}
||T_n(u)-T(u)|| &= \sup_{v \in X_K(\Omega),\,||v||=1} \Big| \int_{\Omega} \rho_n(x) uv\,dx - 
\int_{\Omega} \rho(x) uv\,dx \Big| \\
&\leq \sup_{v \in X_K(\Omega),\,||v||=1} ||\rho_n - \rho||_{\infty} ||u||_2 ||v||_2 \leq C ||\rho_n - \rho||_{\infty},\end{align*}
and the latter tends to $0$ as $n\to\infty$. Now fix $k>0$: reasoning as in \cite[Theorem 2.3.1]{H}, we have for all $n\in\N$
\begin{equation}
|\lambda_{k}^{-1}(\rho_n) - \lambda_{k}^{-1}(\rho)| \leq ||T_n-T||_{\mathcal{L}(X_K(\Omega))}.
\label{T-}
\end{equation}
Indeed, recalling \eqref{E+}, there exists $F \in \mathcal{F}_k$ s.t.\
$$\lambda_{k}^{-1}(\rho)= \inf_{u \in F,\,||u||=1} \int_{\Omega} \rho(x) u^2\,dx.$$
By compactness, there exists $\hat{u} \in F, \; ||\hat{u}||=1$ s.t.\
$$\int_{\Omega} \rho_n(x) \hat{u}^2\,dx= \inf_{u \in F,\,||u||=1} \int_{\Omega} \rho_n(x) u^2\,dx.$$
So we have for all $n\in\N$
\begin{align*}
\lambda_{k}^{-1}(\rho) - \lambda_{k}^{-1}(\rho_n) &\leq \inf_{u \in F,\,||u||=1} \int_{\Omega} \rho(x) u^2\,dx - \inf_{u \in F,\,||u||=1} \int_{\Omega} \rho_n(x) u^2\,dx \\
&\leq \int_{\Omega} \rho(x) \hat{u}^2\,dx - \int_{\Omega} \rho_n(x) \hat{u}^2\,dx
=\left\langle T(\hat{u})-T_n(\hat{u}), \hat{u}\right\rangle \leq ||T-T_n||_{\mathcal{L}(X_K(\Omega))}.
\end{align*}
An analogous argument leads to
$$\lambda_{k}^{-1}(\rho) - \lambda_{k}^{-1}(\rho_n) \geq - ||T-T_n||_{\mathcal{L}(X_K(\Omega))},$$
proving \eqref{T-}. Now \eqref{T}, \eqref{T-} imply $\lambda_k(\rho_n) \rightarrow \lambda_k(\rho)$ as $n\to\infty$.
\end{proof}

\begin{Oss}
In fact, continuous dependence can be proved even with respect to weaker types of convergence, such as weak* convergence of the weights (see \cite[Theorem 3.1]{ACF}). Anyway, continuity in the norm topology is enough for our purposes.
\end{Oss}

\section{Strict monotonicity and u.c.p.}\label{sec3}

\noindent
This section is devoted to proving our main result, i.e., the equivalence between strict monotonicity of the map $\rho\mapsto\lambda_k(\rho)$ ($k\in\Z_0$) and u.c.p.\ of the eigenfunctions. Our definition of u.c.p.\ is the following:

\begin{Def} \label{ucp}
We say that $\rho \in L^{\infty}(\Omega) \setminus \{0\}$ satisfies u.c.p., if for any eigenfunction $u \in X_K(\Omega)$ of \eqref{P} (with any $\lambda \in \sigma(\rho)$)
$$\big|\big\{u=0\big\}\big|=0.$$	
\end{Def}

\noindent
We follow the approach of \cite{FG}. First we note that, by \eqref{E+} and \eqref{E-}, given $\rho, \tilde{\rho} \in L^{\infty}(\Omega) \setminus \{0\}$,
\begin{equation}
\rho \leq \tilde{\rho} \ \Rightarrow \ \lambda_k(\rho) \geq \lambda_k(\tilde{\rho}) \ \text{for all admissible $k \in \Z_0$.}
\label{M}
\end{equation}
First we prove that u.c.p.\ implies strict monotonicity:

\begin{Teo}\label{DIR}
Let $\rho, \tilde{\rho} \in L^{\infty}(\Omega) \setminus \{0\}$ be s.t.\ $\rho \leq \tilde{\rho}, \; \rho \not\equiv \tilde{\rho}$, and either $\rho$ or $\tilde{\rho}$ satisfies u.c.p. Then, $\lambda_k(\rho) > \lambda_k(\tilde{\rho})$ for all admissible $k \in \Z_0$.	
\end{Teo}
\begin{proof}
Assume $\rho$ has u.c.p., $\rho^+,{ \tilde{\rho}}^+ \not\equiv 0, \; k>0$. By \eqref{E+}, there exists $F \in \mathcal{F}_k$ s.t.\
\begin{equation}
\lambda_{k}^{-1}(\rho) = \inf_{u \in F,\,||u||=1} \int_{\Omega} \rho(x) u^2\,dx.
\label{Lk}
\end{equation}
Fix $u \in F, \; ||u||=1$. Two cases may occur:
\begin{itemize}[leftmargin=0.6cm]
\item[$(a)$] if $u$ is a minimizer in \eqref{Lk}, then $u$ is a $\lambda_{k}(\rho)$-eigenfunction, hence $|\{u=0\}|=0$. So we have $\rho u^2 \leq \tilde{\rho} u^2$, with strict inequality on a subset of $\Omega$ with positive measure, hence
$$\lambda_{k}^{-1}(\rho) = \int_{\Omega} \rho(x) u^2\,dx <  \int_{\Omega} \tilde{\rho}(x) u^2\,dx;$$
\item[$(b)$] if $u$ is not a minimizer in \eqref{Lk}, then 
$$\lambda_{k}^{-1}(\rho) < \int_{\Omega} \rho(x) u^2\,dx \leq  \int_{\Omega} \tilde{\rho}(x) u^2\,dx.$$
\end{itemize}
In both cases, we have  
$$\lambda_{k}^{-1}(\rho) < \int_{\Omega} \tilde{\rho}(x) u^2\,dx.$$
Since $F$ has finite dimension, the set of $u$'s above is compact. Recalling also \eqref{E+} with weight $\tilde{\rho}$, we have 
$$\lambda_{k}^{-1}(\rho) < \inf_{u \in F,\,||u||=1} \int_{\Omega} \tilde{\rho}(x) u^2\,dx \leq \lambda_{k}^{-1}(\tilde{\rho}) .$$
Now we assume $\rho^-,\tilde\rho^-\not\equiv 0$ and consider negative eigenvalues, i.e., $k<0$. Set $j=-k$ for simplicity. By \eqref{E-},  there exists $F \in \mathcal{F}_{j-1}$ s.t.\
$$\lambda_{-j}^{-1}(\rho) = \inf_{u \in F^{\perp},\,||u||=1} \int_{\Omega} \rho(x) u^2\,dx.$$
Arguing as above, we see that for all $u \in F^{\perp}$, $||u||=1$
\begin{equation}
\lambda_{-j}^{-1}(\rho)<\int_{\Omega} \tilde{\rho}(x) u^2\,dx.
\label{L-k}
\end{equation}
But $F^{\perp}$ is infinite dimensional, so we can not easily minimize in \eqref{L-k}. Set
$$m= \inf_{u \in F^{\perp},\,||u||=1}  \int_{\Omega} \tilde{\rho}(x) u^2\,dx,$$
and argue by contradiction, assuming $m=\lambda_{-j}^{-1}(\rho)<0$. Let $(u_n)$ be a sequence in $F^{\perp}$, s.t.\ $||u_n||=1$, and 
$$ \int_{\Omega} \tilde{\rho}(x) u_n^2\,dx \rightarrow m.$$
Since $(u_n)$ is bounded, passing if necessary to a subsequence we find $u \in F^{\perp}$ s.t.\ $u_n \rightharpoonup u$ in $X_K(\Omega)$, $u_n \rightarrow u$ in $L^2(\Omega)$. The last relation implies
$$ \int_{\Omega} \tilde{\rho}(x) u^2\,dx = m,$$ 
in particular $u\neq 0$. Set $\hat{u}=u/||u|| \in F^{\perp}$, then $||\hat{u}||=1$, which by \eqref{L-k} implies
$$\lambda_{-j}^{-1}(\rho) < \int_{\Omega} \tilde{\rho}(x) \hat{u}^2\,dx = \frac{m}{||u||^2} = \frac{\lambda_{-j}^{-1}(\rho)}{||u||^2},$$
hence (recalling that $\lambda_{-j}(\rho)<0$) we get $||u||>1$. But $u_n \rightharpoonup u$ in $X_K(\Omega)$ implies $\|u\|\le 1$, a contradiction. So $m > \lambda_{-j}^{-1}(\rho)$, which by \eqref{E-} implies 
$$\lambda_{-j}^{-1}(\rho) < \inf_{u \in F^{\perp},\,||u||=1}  \int_{\Omega} \tilde{\rho}(x) u^2\,dx \leq \lambda_{-j}^{-1}(\tilde{\rho}),$$
so $\lambda_{-j}(\rho) > \lambda_{-j}(\tilde{\rho})$. 
\end{proof}

\noindent
The next result establishes the reverse implication:

\begin{Teo}\label{REV}
Let $\rho \in L^{\infty}(\Omega) \setminus \{0\}$ do not satisfy u.c.p. Then, there exist $\tilde{\rho} \in L^{\infty}(\Omega) \setminus \{0\}$ s.t.\ either $\rho \leq \tilde{\rho}$ or $\rho \geq \tilde{\rho}$, $\rho \not\equiv \tilde{\rho}$, and $k \in \Z_0$ s.t.\ $\lambda_k(\rho) = \lambda_k(\tilde{\rho})$.	
\end{Teo}
\begin{proof}
By Definition \ref{ucp}, we can find $k \in \Z_0$ and a $\lambda_{k}$-eigenfunction $u \in X_K(\Omega)$ s.t.\ $|A|>0$, where $A:=\{u=0\}$. First assume $\rho^+ \not\equiv 0, \; k>0$, and without loss of generality $\lambda_{k}(\rho) < \lambda_{k+1}(\rho)$. For all $\eps \in \R$ set
\[
\rho_{\eps}(x)=
\begin{cases}
\rho(x)                    & \text{if $x \in \Omega \setminus A$} \\
\rho(x) + \eps    & \text{if $x \in A$,}
\end{cases}
\]
so $\rho_{\eps} \in L^{\infty}(\Omega)$ and $\rho_{\eps}\rightarrow \rho$ in $L^{\infty}(\Omega)$ as $\eps \rightarrow 0$.
By Proposition \ref{CD}
$$\lim_{\eps \rightarrow 0} \lambda_{k+1}(\rho_{\eps})= \lambda_{k+1}(\rho) > \lambda_k(\rho),$$
so we can find $\eps \in (0,1)$ s.t.\ $\lambda_{k+1}(\rho_{\eps}) > \lambda_k(\rho)$. Set $\tilde{\rho}= \rho_{\eps} \in L^{\infty}(\Omega) \setminus \{0\}$, so $\rho\le\tilde\rho$, $\rho\not\equiv\tilde\rho$. For all $v \in X_K(\Omega)$ we have
$$\left\langle u,v\right\rangle = \lambda_{k}(\rho) \int_{\Omega} \rho(x) uv\,dx = \lambda_{k}(\rho) \int_{\Omega} \tilde{\rho}(x) uv\,dx,$$
so $\lambda_{k}(\rho) \in \sigma(\tilde{\rho})$ with associated eigenfunction $u$. We can find $h \in \N_0$ s.t.\ 
$$\lambda_{k}(\rho)=\lambda_{h}(\tilde{\rho}) < \lambda_{h+1}(\tilde{\rho}),$$
in particular $\lambda_{h}(\tilde{\rho}) < \lambda_{k+1}(\tilde{\rho})$, which implies $h \leq k$. 
Besides, by \eqref{M} we have
$$\lambda_{k}(\tilde{\rho}) \leq \lambda_{k}(\rho) = \lambda_{h}(\tilde{\rho}),$$
hence $k \leq h$. Summarizing, $h=k$, thus $\lambda_{k}(\rho)=\lambda_k(\tilde{\rho})$. \\
\vskip2pt
\noindent
Now assume $\rho^- \not\equiv 0$ and $k<0$. Set $j=-k$, for simplicity of notation, and  
without loss of generality $\lambda_{-j-1}(\rho) < \lambda_{-j}(\rho)$. Arguing as above (with $\eps<0$), we find $\tilde{\rho} \in L^{\infty}(\Omega)\setminus \{0\}$ s.t.\ $\tilde{\rho} \leq \rho$,
$\tilde{\rho} \not\equiv \rho$, and $\lambda_{-j-1}(\tilde{\rho}) < \lambda_{-j}(\rho)$. For all $v \in X_K(\Omega)$ we have
$$\left\langle u,v\right\rangle = \lambda_{-j}(\rho) \int_{\Omega} \rho(x) uv\,dx = \lambda_{-j}(\rho) \int_{\Omega} \tilde{\rho}(x) uv\,dx,$$
so there exists $i \in \N_0$ s.t.\ $\lambda_{-j}(\rho)= \lambda_{-i}(\tilde{\rho})$, with $u$ as an associated eigenfunction. By $\lambda_{-i}(\tilde{\rho})> \lambda_{-j-1}(\tilde{\rho})$ we have $-i \geq -j$, while by \eqref{M} we have 
$$\lambda_{-j}(\tilde{\rho})\geq \lambda_{-j}(\rho)=\lambda_{-i}(\tilde{\rho}),$$
hence $-j \geq -i$. Thus $-i=-j$ and  $\lambda_{-j}(\rho)= \lambda_{-j}(\tilde{\rho})$.
Clearly, if $\rho$ has constant sign only one of the previous argument applies.
\end{proof}

\begin{Oss}
A partial result for Theorem \ref{DIR} was given in \cite[Proposition 2.10]{IP} for the fractional Laplacian, with two positive weights one of which is in $C^1(\Omega)$.
\end{Oss}

\section{Unique continuation for non-local operators}\label{sec4}

\noindent
This final section is devoted to a brief survey on recent results on u.c.p.\ for non-local operators. Browsing the literature, many results of this type are encountered, dealing in most cases with the fractional Laplacian $(-\Delta)^s$ (which, as seen before, corresponds to our $L_K$ with the kernel $K(x)=|x|^{-N-2s}$). First we recall the main notions of u.c.p.\ considered in the literature:

\begin{Def}\label{SW}
Let $\Omega\subseteq\R^N$ be a domain and $\mathcal{S}$ a family of measurable functions on $\Omega$:
\begin{enumroman}
\item\label{SW1} $\mathcal{S}$ satisfies the strong unique continuation property (s.u.c.p.), if no function $u\in\mathcal{S}\setminus\{0\}$ has a zero of infinite order in $\Omega$;
\item\label{SW2} $\mathcal{S}$ satisfies the unique continuation property (u.c.p.), if no function $u\in\mathcal{S}\setminus\{0\}$ vanishes on a subset of $\Omega$ with positive measure;
\item\label{SW3} $\mathcal{S}$ satisfies the weak unique continuation property (u.c.p.), if no function $u\in\mathcal{S}\setminus\{0\}$ vanishes on an open subset of $\Omega$.
\end{enumroman}
\end{Def}

\noindent
Definition \ref{ucp} corresponds to the case \ref{SW2}. We recall that a function $u\in L^2(\Omega)$ has a zero of infinite order at $x_0\in\Omega$ if for all $n\in\N$
\[\int_{B_r(x_0)}u^2\,dx=O(r^n) \ \text{as $r\to 0^+$.}\]
The relations between the properties depicted in Definition \ref{SW} are the following:
\[\text{\rm s.u.c.p.\ or u.c.p.} \ \Rightarrow {\rm w.u.c.p.}\]
We recall now some recent results on non-local unique continuation.
\vskip2pt
\noindent
In \cite{FF}, Fall and Felli consider fractional Laplacian equations involving regular, lower order perturbations of a Hardy-type potential, of the following type:
\[(-\Delta)^su-\frac{\lambda}{|x|^{2s}}u=h(x)u+f(x,u) \ \text{in $\Omega$,}\]
where $\Omega\subset\R^N$ is a bounded domain s.t.\ $0\in\Omega$, $0<s<\min\{1,N/2\}$, $\lambda<2^{2s}\Gamma^2(\frac{N+2s}{4})/\Gamma^2(\frac{N-2s}{4})$, and $h\in C^1(\Omega\setminus\{0\})$, $f\in C^1(\Omega\times\R)$ satisfy the estimates
\[|h(x)|+|x\cdot\nabla h(x)|\lesssim |x|^{-2s+\eps} \ (\eps>0),\]
\[|f(x,t)t|+|\partial_t f(x,t)t^2|+|\nabla_x F(x,t)\cdot x|\lesssim |t|^p \ \Big(2<p<\frac{2N}{N-2s}\Big),\]
where $F(x,\cdot)$ is the primitive of $f(x,\cdot)$. The main results asserts that, if $u$ is a solution of the equation above and $u$ vanishes of infinite order at $0$, then $u\equiv 0$ (s.u.c.p.). The proof relies on the Caffarelli-Silvestre extension operator, exploited in order to define an adapted notion of frequency function, admitting a limit as $r\to 0^+$.
\vskip2pt
\noindent
Another result of Fall and Felli \cite{FF2} deals with a relativistic Schr\"odinger equation involving a fractional perturbation of $(-\Delta)^s$ and an anisotropic potential:
\[(-\Delta+m^2)^s u-a\Big(\frac{x}{|x|}\Big)\frac{u}{|x|^{2s}}-h(x)u=0 \ \text{in $\R^N$,}\]
where $\Omega$, $s$ are as above, $m\ge 0$, $a\in C^1(S^{N-1})$ and $h\in\Omega$ satisfies a similar estimate. The authors give a precise description of the asymptotic behavior of solutions near the origin, and deduce again s.u.c.p. These results do not apply in our framework, even restricting ourselves to the fractional Laplacian, since they involve {\em smooth} weight functions, differentiability being required in order to derive Pohozaev-type identities.
\vskip2pt
\noindent
Instead, Seo \cite{S13a} considers possibly non-smooth weights in the fractional inequality
\[|(-\Delta)^s u|\le |V(x)u| \ \text{in $\R^N$,}\]
where $N\ge 2$, $N-1\le 2s<N$, and the weight function $V$ satisfies
\[\lim_{r\to 0^+}\sup_{x\in\R^N}\int_{B_r(x)}\frac{|V(y)|}{|x-y|^{N-2s}}\,dy=0.\]
By means of strong Carleman estimates, the author proves w.u.c.p.\ for solutions of the above inequality with $u,(-\Delta)^su\in L^1(\R^N)$. Moreover, Seo \cite{S} obtained a special u.c.p.\ result for potentials $V$ in Morrey spaces.
\vskip2pt
\noindent
We also mention the work of Yu \cite{Y}, where s.u.c.p.\ is proved for fractional powers of linear elliptic operators with Lipschitz continuous coefficients. We recall that, whenever $L$ is a uniformly elliptic operator defined on a bounded domain $\Omega\subset\R^N$, endowed with a discrete set of eigenpairs $(\lambda_k,e_k)$, its $s$-power ($s\in(0,1)$) is defined by
\[L^s u=\sum_{k\in\Z}\lambda_k^s u_k e_k,\]
where $u=\sum u_k e_k$ is the expansion of $u$ in the orthonormal basis $(e_k)$. Such construction, with $L=-\Delta$, leads to the definition of the {\em spectral} fractional Laplacian. Note that such operator, in bounded domains, does not coincide with the Dirichlet fractional Laplacian $(-\Delta)^s$, as observed in \cite{SV1}, since the first eigenvalue (with weight $1$) of the spectral fractional Laplacian is greater than that of $(-\Delta)^s$. We note that the same can be seen by comparing eigenfunctions, as those of the spectral fractional Laplacian lie in $C^1(\overline\Omega)$, while those of $(-\Delta)^s$ have optimal regularity $C^s(\overline\Omega)$.
\vskip2pt
\noindent
The problem of {\em non-smooth} weights is the focus of the work of R\"uland \cite{R}, dealing with the fractional Schr\"odinger-type equation
\[(-\Delta)^s u=V(x)u \ \text{in $\R^N$,}\]
with a measurable function $V=V_1+V_2$ satisfying
\[V_1(x)=|x|^{-2s}h\Big(\frac{x}{|x|}\Big) \ (h\in L^\infty(S^{N-1})), \ |V_2(x)|\le c|x|^{-2s+\eps} \ (c,\eps>0).\]
For $s<1/2$, the following additional conditions are assumed: either $V_2\in C^1(\R^N\setminus\{0\})$ satisfies $|x\cdot\nabla V_2(x)|\lesssim |x|^{-2s+\eps}$, or $s\ge 1/4$ and $V_1\equiv 0$. Under such assumptions, any solution $u\in H^s(\R^N)$ vanishing of infinite order at $0$ is in fact $u\equiv 0$ (s.u.c.p.). R\"uland's approach, based on Carleman estimates, allows for non-smooth weights and generalization to anisotropic operators. A w.u.c.p.\ result for $(-\Delta)^s$ ($s\in(0,1)$), as well as s.u.c.p.\ for the square root of the Laplacian $(-\Delta)^{1/2}$, with a weight in $L^{N+\eps}(\R^N)$, appear in R\"uland \cite{R0}.
\vskip2pt
\noindent
The result of Ghosh, R\"uland, Salo, and Uhlmann \cite[Theorem 3]{GRSU} is the closest to our framework. For any $V\in L^\infty(\Omega)$ and any $s\in[1/4,1)$, if $u\in H^s(\R^N)$ solves
\[(-\Delta)^su=V(x)u \ \text{in $\Omega$}\]
and vanishes on a subset of $\Omega$ with positive measure, then $u\equiv 0$ (u.c.p.). Here the approach is based on Carleman estimates again, along with a boundary u.c.p.\ for solutions of the (local) degenerate elliptic equation
\[\nabla\cdot\big(x_{N+1}^{1-2s}\nabla u\big)=0 \ \text{in $\R^{N+1}_+$,}\]
with homogeneous Robin conditions. By combining the results of \cite{GRSU} with our Theorems \ref{DIR}, then, we have:

\begin{Cor}
Let $L_K$ be defined by $s\in[1/4,1)$ and $K(x)=|x|^{-N-2s}$, $\rho,\tilde\rho\in L^\infty(\Omega)$ be s.t.\ $\rho\le\tilde\rho$, $\rho\not\equiv\tilde\rho$. Then, $\lambda_k(\rho)>\lambda_k(\tilde\rho)$ for all admissible $k\in\Z_0$.
\end{Cor}
\vskip4pt
\noindent
{\small {\bf Acknowledgement.} Both authors are members of GNAMPA (Gruppo Nazionale per l'Analisi Matematica, la Probabilit\`a e le loro Applicazioni) of INdAM (Istituto Nazionale di Alta Matematica 'Francesco Severi'). A.\ Iannizzotto is partially supported by the research project {\em Integro-differential Equations and Non-Local Problems}, funded by Fondazione di Sardegna (2017).}

\end{document}